\documentclass[a4paper]{article}

\usepackage[top=3cm,bottom=3cm,left=3cm,right=3cm]{geometry}

\usepackage{palatino}
\usepackage{latexsym}

\usepackage[mathscr]{eucal}
\usepackage{amsmath}
\usepackage{amsthm}
\usepackage{amsfonts}
\usepackage{amssymb}
\usepackage{amscd}
\usepackage{color}
\usepackage{graphicx}
\usepackage{graphics}
\usepackage{pifont}
\usepackage{subfigure}
\usepackage[makeroom]{cancel}
\usepackage[normalem]{ulem}
\usepackage[dvipsnames]{xcolor}
\usepackage{tikz}

\theoremstyle{plain}
\newtheorem{theorem}{Theorem}

\newtheorem{lemma}[theorem]{Lemma}
\newtheorem{proposition}[theorem]{Proposition}

\title{Small-time global approximate controllability of bilinear wave equations}

\begin{document}
\author{Eugenio Pozzoli\footnote{Dipartimento di Matematica, Università di Bari, I-70125 Bari, Italy. \emph{Email address}: eugenio.pozzoli@uniba.it}}
\maketitle
\abstract{
We consider a bilinear control problem for the wave equation on a torus of arbitrary dimension. We show that the system is globally approximately controllable in arbitrarily small times from a dense family of initial states. The control strategy is explicit, and based on a small-time limit of conjugated 
dynamics to move along non-directly accessible directions (a.k.a. Lie brackets of the generators). 
}\\ 

\section{Introduction}
\subsection{The model}
In this paper we study the following bilinear wave equation on a $d$-dimensional torus $\mathbb{T}^d$, $d\in \mathbb{N}$,
\begin{equation}\label{eq:wave}
\frac{\partial^2}{\partial t^2}w(x,t)=\Big(\Delta+\mu(x,t)\Big)w(x,t),\quad (x,t)\in\mathbb{T}^d\times\mathbb{R},
\end{equation}
where $(w(\cdot,t),\frac{\partial}{\partial t}w(\cdot,t))\in H^1\times L^2(\mathbb{T}^d,\mathbb{R})$ is the state of the system (describing the profile $w$ and the velocity $\frac{\partial}{\partial t} w$ of the wave), $\Delta=\sum_{i=1}^d\frac{\partial^2}{\partial x_i^2}$ is the Laplacian, and $\mu$ is a function which plays the role of the control. In particular, we control the system through \emph{low modes forcing}: this means that we assume that the control function $\mu$ can be written as
\begin{equation}\label{eq:low-modes}
\mu(x,t)=\sum_{j=0}^{2d}p_j(t)\mu_j(x),
\end{equation}
where $p=(p_0,\dots,p_{2d})$ are piecewise constant control laws that can be freely chosen, and the $\mu_j$ are fixed to be the first real Fourier modes of the system:
\begin{equation}\label{eq:cos-sin}
(\mu_0(x),\dots,\mu_{2d}(x)):=(1,\cos(e_1x),\sin(e_1x),\dots,\cos(e_dx),\sin(e_dx)),
\end{equation}
where $\{e_i,i=1,\dots,d\}\subset\mathbb{Z}^d$ is the standard basis of $\mathbb{R}^d$. The dependence of the state on the control is nonlinear, and hence \eqref{eq:wave} is a nonlinear control problem in infinite dimensions. 


\subsection{The main result}
Determining the minimal time needed for the global approximate controllability of bilinear PDEs is a fundamental problem, whose answer is known in very few cases (see, e.g., \cite{minimal-time-thomas,minimal-time-coron,minimal-time-approximate}). The scope of this paper is proving that, in the case of wave equations, this minimal time is zero. 

More precisely, our main result is the small-time global controllability of \eqref{eq:wave}, approximately in $H^1\times L^2(\mathbb{T}^d)$, from any nonzero initial state whose profile has a finite number of non-vanishing Fourier modes. 
\begin{theorem}\label{thm:main-result}
Consider an initial state $(0,0)\neq(w_0,\dot{w}_0)\in H^1\times L^2(\mathbb{T}^d)$ such that
\begin{equation}\label{eq:initial-condition1}
w_0\neq 0, \quad \langle w_0,e^{ikx}\rangle_{L^2}=0\, \text{ for all but a finite set of } k\in \mathbb{Z}^d,
\end{equation}
or
\begin{equation}\label{eq:initial-condition2}
w_0= 0,\dot{w}_0\neq0, \quad \langle \dot{w}_0,e^{ikx}\rangle_{L^2}=0\, \text{ for all but a finite set of } k\in \mathbb{Z}^d.
\end{equation}
Then, for any final state $(w_1,\dot{w}_1)\in H^1\times L^2(\mathbb{T}^d)$ and any error and time $\varepsilon,T>0$, there exists a piecewise constant control law $p:[0,T]\to \mathbb{R}^{2d+1}$ such that the solution $w$ of \eqref{eq:wave} associated with the control \eqref{eq:low-modes},\eqref{eq:cos-sin} and with the initial condition $\left(w(t=0),\frac{\partial}{\partial t}w(t=0)\right)=(w_0,\dot{w}_0)$ satisfies
 $$\left\|\left(w(\cdot,T),\frac{\partial}{\partial t}w(\cdot,T)\right)-(w_1,\dot{w}_1)\right\|_{H^1\times L^2(\mathbb{T}^d)}< \varepsilon. $$
\end{theorem}
Let us stress that the initial state $(0,0)$ is an equilibrium of system \eqref{eq:wave} regardless of the control choice, and hence system \eqref{eq:wave} cannot be steered anywhere starting from $(0,0)$. We also remark that assumptions \eqref{eq:initial-condition1} or \eqref{eq:initial-condition2} on the initial state are technical but their necessity is an open question.
 In view of the strictly positive minimal time needed for the local exact controllability of bilinear wave equations showed by Beauchard in \cite{beauchard-wave}, Theorem \ref{thm:main-result} may seem surprising and relies upon the approximate nature of the small-time controllability result.
\subsection{The technique}

It is convenient to recast \eqref{eq:wave}, \eqref{eq:low-modes} as a first order evolution equation in the state $W=(w,\frac{\partial}{\partial t}w)$: this gives the system
\begin{equation}\label{eq:bilinear}
\frac{\partial}{\partial t}W(x,t)=\left(\mathcal{A}+\sum_{j=0}^{2d}p_j(t)\mu_j(x)\mathcal{B}\right)W(x,t),\quad (x,t)\in\mathbb{T}^d\times \mathbb{R},
\end{equation}
where
\begin{equation}\label{eq:matrices}
\mathcal{A}=\begin{pmatrix}
0 & I\\
\Delta & 0
\end{pmatrix},\quad \mathcal{B}=\begin{pmatrix}
0 & 0\\
I & 0
\end{pmatrix}. 
\end{equation}


The proof of Theorem \ref{thm:main-result} is based on the following small-time limit of conjugated dynamics, holding for any initial condition $(w_0,\dot{w_0})\in H^1\times L^2(\mathbb{T}^d)$:
\begin{equation}\label{eq:wave-limit}
 \lim_{\tau \to 0}e^{-\tau^{-1/2}\mu_j\mathcal{B}}e^{\tau \mathcal{A}}e^{\tau^{-1/2}\mu_j\mathcal{B}}\begin{pmatrix} w_0\\
\dot{w}_0
\end{pmatrix}
=\begin{pmatrix}
w_0\\
\dot{w}_0-\mu_j^2w_0
\end{pmatrix}.
\end{equation}
This limit has been introduced by Duca and Nersesyan in \cite{duca-nersesyan} for controlling nonlinear Schr\"odinger equations. It can be thought as the following explicit control strategy: 
we first apply an impulsive control with amplitude $\tau^{-1/2}$ (which gives the evolution $e^{\tau^{-1/2}\mu_j\mathcal{B}}$), we then let the system evolve freely for a time interval of size $\tau$ (which gives the evolution $e^{\tau \mathcal{A}}$), and we finally apply again an impulsive control with amplitude $-\tau^{-1/2}$ (which gives the evolution $e^{-\tau^{-1/2}\mu_j\mathcal{B}}$). The control law associated to this strategy is depicted in Figure \ref{fig:opposite-kicks}.

 \begin{figure}[ht!]\begin{center}
\includegraphics[width=0.3\linewidth, draft = false]{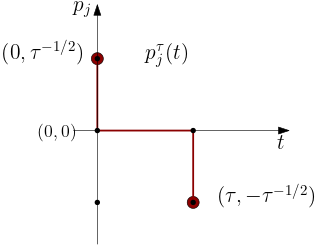}
\caption{\small The control law yielding the limiting propagator of \eqref{eq:wave-limit} can be thought as a fractional derivative of a Dirac delta. Indeed, it can be thought as the control law $p^\tau_j(t)=\frac{1}{\tau^{1/2}}(\delta(t)-\delta(t-\tau))$ with $\tau\to 0$, where $\delta(s)$ is a Dirac delta centred at $s=0$. } \label{fig:opposite-kicks} 
\end{center}
\end{figure}

As observed in the work of the author in collaboration with Chambrion \cite{small-time-molecule} on Schr\"odinger equations, the limiting dynamic \eqref{eq:wave-limit} (also in this case of wave equations) corresponds to the exponential of an iterated Lie bracket between $\mathcal{A}$ and $\mu_j\mathcal{B}$: more precisely, limit \eqref{eq:wave-limit} can also be written as
\begin{equation}\label{eq:abstract-limit}
 \lim_{\tau \to 0}e^{-\tau^{-1/2}\mu_j\mathcal{B}}e^{\tau \mathcal{A}}e^{\tau^{-1/2}\mu_j\mathcal{B}}\begin{pmatrix} w_0\\
\dot{w}_0
\end{pmatrix}=\exp\left(\frac{1}{2}[[\mathcal{A},\mu_j\mathcal{B}],\mu_j\mathcal{B}]\right)\begin{pmatrix} w_0\\
\dot{w}_0
\end{pmatrix},
\end{equation}
where 
$$  [[\mathcal{A},\mu_j\mathcal{B}],\mu_j\mathcal{B}]=\begin{pmatrix}
0 & 0\\
-2\mu_j^2 & 0
\end{pmatrix},$$
and the Lie bracket $[\mathcal{C},\mathcal{D}]$ of two linear operators $\mathcal{C},\mathcal{D}$ is formally defined as the commutator $\mathcal{C}\mathcal{D}-\mathcal{D}\mathcal{C}$. It is then interesting to interpret Theorem \ref{thm:main-result} as a consequence of a geometric control technique adapted to this infinite dimensional setting: as in finite-dimensional nonlinear control systems \cite{Lobry,sussmann-jurje}, one can think of the generators $\mathcal{A}$ and $\mu_j\mathcal{B}$ as directions that are directly accessible to the system, and recover from them additional directions that were not directly accessible, as for instance the Lie bracket $[[\mathcal{A},\mu_j\mathcal{B}],\mu_j\mathcal{B}]$. 
Moreover, the exponential flow computed on these directions (and applied to some initial condition) describes states that are approximately reachable in arbitrarily small times. 





\subsection{Related literature}
As shown in \cite{BMS,Chambrion-Caponigro-Boussaid-2020}, bilinear PDEs are never exactly controllable in the larger functional space where the evolution is defined. In particular, system \eqref{eq:wave} is not exactly controllable in $H^1\times L^2$ (we refer also to \cite{chambrion-laurent} for an analogous obstructions to the exact controllability of \eqref{eq:wave} even in the presence of a state nonlinearity). Researchers have then focus their efforts on the approximate controllability, or the exact controllability in smaller functional spaces, of bilinear PDEs.

The global approximate controllability in $H^1\times L^2$ of bilinear wave equations of the form \eqref{eq:wave} has been firstly proved by Ball, Marsden, and Slemrod in \cite{BMS}, on a 1-D interval with Dirichlet boundary conditions, with space-independent control function $\mu(x,t)=\mu(t)$, from any initial condition $(w_0,\dot{w}_0)$ whose Fourier modes are all non-vanishing, and in times $T\geq 1$. 


The main contributions of Theorem \ref{thm:main-result}, in this sense, are extensions of the global approximate controllability result in two ways: it is proved to hold in any space dimensions, and in arbitrarily small times. Moreover, the control technique behind the proof of Theorem \ref{thm:main-result} has the advantage of being explicit with the use of piecewise constant (in time) control laws; also, it suggests a useful link between controllability of bilinear PDEs and Lie brackets of the generators.


Approximate controllability of bilinear wave equations (on a 1-D interval with Dirichlet boundary conditions) has also been studied by Khapalov in \cite{khapalov-wave2}, where it is shown in $H^1\times L^2$, from any nonzero initial state, towards any state of the form $(w_1,0)$, in large times (and a similar result is obtained in \cite{khapalov-wave1} even in the presence of state nonlinearity).

The local exact controllability of bilinear wave equations (on a 1-D interval with Neumann boundary conditions) has been studied by Beauchard in \cite{beauchard-wave}: around the state $(w_0,\dot{w}_0)=(1,0)$, it is proved to hold (in the optimal space $H^3\times H^2$) if and only if the control time satisfies $T>2$ (similar conclusions are obtained in \cite{laurent} in the presence of state nonlinearity, and in the recent work \cite{urbani-wave} in the case of a degenerate Laplacian).

We conclude this bibliographical review by commenting on the geometric control technique of low mode forcing and the small-time controllability analysis of bilinear PDEs: the assumption on the spatial control to be supported only on a finite number of Fourier modes (cf. \eqref{eq:low-modes}) originates in the work of Agrachev and Sarychev \cite{navier-stokes,agrachev2} and Shirikyan \cite{Shirikyan1,Shirikyan2} on the additive control of Navier-Stokes equations. More recently, it has been introduced in the bilinear setting for studying small-time controllability properties of Schr\"odinger equations by Duca and Nersesyan \cite{duca-nersesyan,duca-nersesyan2} and by Coron, Xiang, and Zhang \cite{coron-small-semiclassical}. 

The results on bilinear wave equations obtained in this paper also testify about the versatility of the control strategy of small-time conjugated dynamics, readapted from bilinear Schr\"odinger equations \cite{duca-nersesyan,small-time-molecule}. In a forthcoming work \cite{duca-io-urbani}, we will show how this strategy also furnishes new results on the small-time approximate controllability of bilinear heat equations.

\subsection{Structure of the paper}
The paper is organized as follow: in Section \ref{sec:pre} we show a preliminary result on the free evolution; in Section \ref{sec:conj-dyn} we prove the small-time limit \eqref{eq:wave-limit} of conjugated dynamics; in Section \ref{sec:saturation} we show a density property of trigonometric functions; in Section \ref{sec:velocity} we prove that the velocity of the wave profile is globally approximately controllable in small times; finally, in Section \ref{sec:proof} we put things together and conclude the proof of Theorem \ref{thm:main-result}.

\section{Preliminaries}\label{sec:pre}
Let 
\begin{equation}\label{eq:eigenfunctions}
\varphi_k(x)=(2\pi)^{-d/2}e^{ikx},\quad\lambda_k=|k|^2,\quad k\in \mathbb{Z}^d,
\end{equation}
be the eigenfunctions and eigenvalues of $-\Delta$ on $\mathbb{T}^d$ which satisfy 
\begin{equation}\label{eq:spectral-dec}
-\Delta \phi_k=\lambda_k\phi_k,\quad \langle\varphi_k,\varphi_j\rangle=\begin{cases}
1,&j=k,\\
0,&j\neq k,
\end{cases} \quad k,j\in \mathbb{Z}^d,
\end{equation}
 where $\langle\cdot,\cdot\rangle$ denotes the scalar product of $L^2(\mathbb{T}^d)$. The domains of the linear operators $\Delta$ and $\mathcal{A}$ are respectively $H^2(\mathbb{T}^d) $ and $H^2\times H^1(\mathbb{T}^d)$, where
 $$H^k(\mathbb{T}^d)=\left\{\phi\in L^2(\mathbb{T}^d)\mid \sum_{j\in\mathbb{Z}}\lambda_j^{k/2}|\langle \phi,\varphi_j\rangle|^2<\infty\right\}.$$
Using the spectral decomposition \eqref{eq:spectral-dec} of $\Delta$, for any $(w_0,\dot{w}_0)\in H^1\times L^2(\mathbb{T}^d)$ and $t\in\mathbb{R}$, one can write $(w(t),\frac{\partial}{\partial t} w(t)):=e^{t\mathcal{A}}(w_0,\dot{w}_0)$ as a converging series in $H^1\times L^2(\mathbb{T}^d)$:
\begin{align}
w(t)&=\left(\langle w_0,\varphi_0\rangle+\langle \dot{w}_0,\varphi_0\rangle t\right)\varphi_0+\sum_{k\in \mathbb{Z}^d\setminus \{0\}} \left(\langle w_0,\varphi_k\rangle\cos(\sqrt{\lambda_k}t)+\langle \dot{w}_0,\varphi_k\rangle\frac{\sin(\sqrt{\lambda_k}t)}{\sqrt{\lambda_k}} \right)\varphi_k,\label{eq:free-dyn1} \\
\frac{\partial}{\partial t} w(t)&=\langle \dot{w}_0,\varphi_0\rangle\varphi_0+\sum_{k\in \mathbb{Z}^d\setminus \{0\}} \left(-\sqrt{\lambda_k}\langle w_0,\varphi_k\rangle\sin(\sqrt{\lambda_k}t)+\langle \dot{w}_0,\varphi_k\rangle\cos(\sqrt{\lambda_k}t) \right)\varphi_k. \label{eq:free-dyn2}
\end{align}
 We start by showing that the free evolution can be used to instantaneously change the initial profile into an arbitrary profile, if the initial velocity is arbitrary. 
\begin{proposition}\label{lem:moment-problem}
For any initial and final profile $w_0,w_1\in H^1(\mathbb{T}^d)$ such that
\begin{equation}\label{eq:finiteness}
\langle w_j,\varphi_k\rangle_{L^2}=0 \text{ for all but a finite set of } k\in \mathbb{Z}^d, \quad j=1,2 \end{equation}
 and any positive time $T>0$, there exist a smaller time $\tau\in[0,T)$ and an initial velocity $f\in L^2(\mathbb{T}^d)$ such that the solution $w$ of \eqref{eq:wave} associated with the identically zero control $p=0$ with initial condition $\left(w(t=0),\frac{\partial}{\partial t}w(t=0)\right)=(w_0,f)$ satisfies $w(\tau)=w_1$ in $H^1(\mathbb{T}^d)$.
\end{proposition}

\begin{proof}
We define
\begin{align*}
f_\tau&=\left(\frac{\langle w_1,\varphi_0\rangle-\langle w_0,\varphi_0\rangle}{\tau}\right)\varphi_0\\
&+\sum_{k\in\mathbb{Z}^d\setminus\{0\}}\left( \frac{\langle w_1,\varphi_k\rangle}{\sin(\sqrt{\lambda_k}\tau)/\sqrt{\lambda_k}}-\frac{\langle w_0,\varphi_k\rangle}{\sin(\sqrt{\lambda_k}\tau)/\sqrt{\lambda_k}}\cos(\sqrt{\lambda_k}\tau) \right)\varphi_k.
\end{align*}
Notice that we can choose $\tau<T$ such that $f_\tau$ is well-defined: indeed, $f$ is a sum on a finite subset of $\mathbb{Z}^d$ (cf. \eqref{eq:finiteness}), and denoting the latter as
 $$\mathcal{K}:=\{k\in\mathbb{Z}^d\mid \langle w_0,\varphi_k\rangle\neq 0,\text{ or } \langle w_1,\varphi_k\rangle\neq 0\},$$ it suffices to choose $\tau<T$ such that 
$$\tau\neq \frac{n \pi}{|k|},\quad \forall n\in\mathbb{N}, k\in\mathcal{K}.  $$
In this way $f=f_\tau$ is well-defined, belongs to $L^2(\mathbb{T}^d)$, and thanks to \eqref{eq:free-dyn1} one easily checks that $w(\tau)=\sum_{k\in \mathbb{Z}^d}\langle w_1,\varphi_k\rangle\varphi_k=w_1$.
\end{proof}


We now recall the well-posedness of \eqref{eq:bilinear},\eqref{eq:cos-sin} (for a proof, see e.g. \cite[Proposition 2]{beauchard-wave}).

\begin{proposition}
Given $T>0$, $p\in L^1([0,T],\mathbb{R}^{2d+1})$, and $W_0=(w_0,\dot{w}_0)\in H^1\times L^2(\mathbb{T}^d)$, there exists a unique weak solution $\mathcal{R}(t,W_0,p)$ of \eqref{eq:bilinear}, that is, a unique function $\mathcal{R}(\cdot,W_0,p)\in C^0([0,T],H^1\times L^2(\mathbb{T}^d))$ satisfying the following equality in $H^1\times L^2(\mathbb{T}^d)$
\begin{equation}\label{eq:weak-solution}
\mathcal{R}(t,W_0,p)=e^{t\mathcal{A}}W_0+\int_{0}^te^{(t-s)\mathcal{A}}\left(\sum_{j=0}^{2d}p_j(s)\mu_j(x)\mathcal{B}\right)\mathcal{R}(s,W_0,p)ds. 
\end{equation}

Moreover, there exists $C=C(\|p\|_{L^1},T)>0$ such that for any other $W_1=(w_1,\dot{w}_1)\in H^1\times L^2(\mathbb{T}^d)$ the following holds
\begin{equation}\label{eq:continuity}
\|\mathcal{R}(\cdot,W_0,p)-\mathcal{R}(\cdot,W_1,p)\|_{C^0([0,T],H^1\times L^2(\mathbb{T}^d))}\leq C\|W_0-W_1\|_{H^1\times L^2(\mathbb{T}^d)}.
\end{equation}
\end{proposition}

\section{Small-time conjugated dynamics}\label{sec:conj-dyn}
In this section we prove limit \eqref{eq:wave-limit}. 
\begin{proposition}\label{prop:limit}
Let $\xi,\psi\in L^\infty(\mathbb{T}^d)$. For any $(w_0,\dot{w}_0)\in H^1\times L^2(\mathbb{T}^d)$, the following limits hold in $H^1\times L^2(\mathbb{T}^d)$
\begin{align}
\lim_{\tau\to 0}\exp\left(\tau \left(\mathcal{A}+\frac{\xi}{\tau}\mathcal{B}\right)\right)\begin{pmatrix}
w_0\\
\dot{w}_0
\end{pmatrix}&=\begin{pmatrix}
w_0\\
\dot{w}_0+\xi w_0
\end{pmatrix}, \label{eq:limit1}\\
\lim_{\tau\to 0}e^{-\tau^{-1/2}\psi\mathcal{B}}e^{\tau \mathcal{A}}e^{\tau^{-1/2}\psi\mathcal{B}}\begin{pmatrix}
w_0\\
\dot{w}_0
\end{pmatrix}&=\begin{pmatrix}
w_0\\
\dot{w}_0-\psi^2w_0
\end{pmatrix}. \label{eq:limit2}
\end{align}
\end{proposition}
\begin{proof}[Proof of \eqref{eq:limit1}]
Using the explicit expressions for the free dynamics \eqref{eq:free-dyn1},\eqref{eq:free-dyn2}, one sees that $\mathcal{A}$ generates a strongly continuous group $\{e^{t \mathcal{A}}\}_{t \in \mathbb{R}}$ of bounded operators on $H^1\times L^2(\mathbb{T}^d)$. Since $\frac{\xi}{\tau}\mathcal{B}$ is bounded for any $\tau\in\mathbb{R}$, the same does $M_\tau:=\mathcal{A}+\frac{\xi}{\tau}\mathcal{B}$ (see, e.g., \cite[Section 2.1]{chambrion-laurent}). Consider then 
$$V_\tau(t):=e^{tM_\tau}W_0=\exp\left(t \left(\mathcal{A}+\frac{\xi}{\tau}\mathcal{B}\right)\right)W_0,$$
i.e. its associated group at time $t$, applied to $W_0=(w_0,\dot{w}_0)\in H^1\times L^2(\mathbb{T}^d)$: we have to prove that $V_\tau(\tau)$ tends to $(w_0,\dot{w}_0+\xi w_0)$ in $H^1\times L^2(\mathbb{T}^d)$ when $\tau \to 0$. We compute $V_\tau(t)$ by considering its Dyson expansion: this is obtained by iterating \eqref{eq:weak-solution}, where the bounded time-dependent perturbative term $\sum_{j=0}^{2d} p_j(t)\mu_j(x)\mathcal{B}$ is replaced with the bounded time-independent perturbative term $\frac{\xi}{\tau}\mathcal{B}$. This procedure gives
\begin{equation}\label{eq:dyson1}
V_\tau(t)=e^{t\mathcal{A}}W_0+\sum_{j=1}^\infty \frac{e^{t\mathcal{A}}B(t)_j}{\tau^{j}},
\end{equation}
where 
$$B(t)_j=\int_0^{t}\!\!\int_0^{t_1}\!\!\!\dots\!\!\int_0^{t_{j-1}}\left(\prod_{i=1}^j e^{-t_i\mathcal{A}}\xi\mathcal{B}e^{t_i\mathcal{A}}\right) dt_1\dots dt_j W_0, $$
and the series \eqref{eq:dyson1} converges in $H^1\times L^2(\mathbb{T}^d)$ (see e.g. \cite[Section 2.1]{chambrion-laurent}, or analogously \eqref{eq:B-fac} below). We claim that
\begin{equation}\label{eq:B-fac}
B(t)_j=\frac{1}{j!}\left(\int_{0}^t  e^{-s\mathcal{A}}\xi\mathcal{B}e^{s\mathcal{A}} ds\right)^j W_0,
\end{equation}
and prove it by induction. The case $j=1$ is obvious. Now, by inductive hypothesis, we can write
$$B(t)_j=\int_0^t\frac{1}{(j-1)!}\left(\int_0^{t_1}e^{-s\mathcal{A}}\xi\mathcal{B}e^{s\mathcal{A}} ds\right)^{j-1}e^{-t_1\mathcal{A}}\xi\mathcal{B}e^{t_1\mathcal{A}} dt_1 W_0. $$
We make the change of variable
$$v(t_1)=\int_0^{t_1}e^{-s\mathcal{A}}\xi\mathcal{B}e^{s\mathcal{A}} ds, $$
which gives
$$B(t)_j=\frac{1}{(j-1)!}\int_0^tv^{j-1}dv W_0=\frac{1}{j!}\left(\int_{0}^t  e^{-s\mathcal{A}}\xi\mathcal{B}e^{s\mathcal{A}} ds\right)^j W_0, $$
and the claim is proved. 
From this, we see that
$$V_\tau(\tau)\to \left(\sum_{j=0}^\infty \frac{(\xi \mathcal{B})^j}{j!}\right)W_0=(I+\xi\mathcal{B})W_0=\begin{pmatrix}
w_0\\
\dot{w}_0+\xi w_0
\end{pmatrix}, \quad\tau \to 0, $$
in $H^1\times L^2(\mathbb{T}^d)$ (where we used that $\mathcal{B}^2=0$), which concludes the proof. 
\end{proof}

\begin{proof}[Proof of \eqref{eq:limit2}]
Since $\mathcal{A}$ generates a strongly continuous group of bounded operators on $H^1\times L^2(\mathbb{T}^d)$ and $\tau^{-1/2}\psi\mathcal{B}$ is bounded for all $\tau>0$, the same does 
$$L_\tau:=e^{-\tau^{-1/2}\psi\mathcal{B}}\mathcal{A}e^{\tau^{-1/2}\psi\mathcal{B}}.$$ 
Consider then $W_\tau(t):=e^{tL_\tau}W_0$, i.e. its associated group at time $t$, applied to $W_0=(w_0,\dot{w}_0)\in H^1\times L^2(\mathbb{T}^d)$: we have
\begin{equation}\label{eq:exp-conjugate-action}
W_\tau(t)=\exp\left(e^{-\tau^{-1/2}\psi\mathcal{B}}t \mathcal{A}e^{\tau^{-1/2}\psi\mathcal{B}}\right)W_0.
\end{equation}
By computing the time derivative of $e^{\tau^{-1/2}\psi\mathcal{B}}W_\tau(t)$, we see that
$$W_\tau(t)=e^{-\tau^{-1/2}\psi\mathcal{B}}e^{t \mathcal{A}}e^{\tau^{-1/2}\psi\mathcal{B}}W_0. $$
To conclude the proof, we are then left to prove that the RHS of \eqref{eq:exp-conjugate-action} computed at time $t=\tau$ tends to $(w_0,\dot{w}_0-\psi^2w_0)$ in $H^1\times L^2(\mathbb{T}^d)$, as $\tau\to 0$.

In order to do so, we first notice that
$$
e^{-\tau^{-1/2}\psi\mathcal{B}}t \mathcal{A}e^{\tau^{-1/2}\psi\mathcal{B}}=\begin{pmatrix}
t\tau^{-1/2}\psi & t\\
-t\tau^{-1}\psi^2+t\Delta & -t\tau^{-1/2}\psi
\end{pmatrix},
$$
where we used that $\mathcal{B}^2=0$ to compute the exponentials, which allows us to write
$$
W_\tau(t)=\exp\left(t\left(\mathcal{A}+\frac{\psi}{\tau^{1/2}}[\mathcal{A},\mathcal{B}]+\frac{\psi^2}{2\tau}[[\mathcal{A},\mathcal{B}],\mathcal{B}]\right)\right)W_0,
$$
where 
$$[\mathcal{A},\mathcal{B}]=\begin{pmatrix}
I & 0\\
0 & -I
\end{pmatrix},\quad [[\mathcal{A},\mathcal{B}],\mathcal{B}]=\begin{pmatrix}
0 & 0\\
-2I & 0
\end{pmatrix}.  $$
Notice moreover that $[\mathcal{A},\mathcal{B}], [[\mathcal{A},\mathcal{B}],\mathcal{B}]$ are bounded. We are thus left to prove that $W_\tau(\tau)$ tends to $(w_0,\dot{w}_0-\psi^2w_0)$ in $H^1\times L^2(\mathbb{T}^d)$, as $\tau\to 0$. We do it by considering the Dyson expansion of $W_\tau(t)$: this is obtained by iterating \eqref{eq:weak-solution}, where the bounded time-dependent perturbative term $\sum_{j=0}^{2d} p_j(t)\mu_j(x)\mathcal{B}$ is replaced with the bounded time-independent perturbative term $\frac{\psi}{\tau^{1/2}}[\mathcal{A},\mathcal{B}]+\frac{\psi^2}{2\tau}[[\mathcal{A},\mathcal{B}],\mathcal{B}]$. This procedure gives
\begin{equation}\label{eq:dyson}
W_\tau(t)=e^{t\mathcal{A}}W_0+\sum_{j=1}^\infty\left( \frac{e^{t\mathcal{A}}}{\tau^{j}}C(t)_{j}+\sum_{i=1}^j\frac{e^{t\mathcal{A}}}{\tau^{j-\frac{i}{2}}}R(t)_{j,i}\right),
\end{equation}
where
$$C(t)_j:=\int_0^{t}\!\!\int_0^{t_1}\!\!\!\dots\!\!\int_0^{t_{j-1}}\left(\prod_{i=1}^j e^{-t_i\mathcal{A}}\frac{\psi^2}{2}[[\mathcal{A},\mathcal{B}],\mathcal{B}]e^{t_i\mathcal{A}}\right) dt_1\dots dt_j W_0, $$
and 
\begin{align*}
R(t)_{j,i}
:=\!\!\!\!\!\!\!\!\!\!\!\!\!\!\!\sum_{\substack{R_{k_1},\dots,R_{k_i}=\psi[\mathcal{A},\mathcal{B}],\{k_1,\dots,k_i\}\subset\{1,\dots,j\}\\
R_{s_{1}},\dots,R_{s_{j-i}}=\frac{\psi^2}{2}[[\mathcal{A},\mathcal{B}],\mathcal{B}],\{s_1,\dots,s_{j-i}\}\cap \{k_1,\dots,k_i\}=\emptyset}}
\!\!\!\!\!\!\!\!\!\!\!\!\!\!\!\!\int_0^t\!\!\int_0^{t_1}\!\!\!\dots\!\!\int_0^{t_{j-1}}\!\!\!\left(\prod_{k=1}^j e^{-t_k\mathcal{A}}R_ke^{t_k\mathcal{A}}\right) dt_1\dots dt_j W_0.
\end{align*}

The series \eqref{eq:dyson} converges in $H^1\times L^2(\mathbb{T}^d)$ (see, e.g., \cite[Section 2.1]{chambrion-laurent}). As in \eqref{eq:B-fac}, we can write
\begin{equation}\label{eq:C-fac}
C_j(t)=\frac{1}{j!}\left(\int_{0}^t  e^{-s\mathcal{A}}\frac{\psi^2}{2}[[\mathcal{A},\mathcal{B}],\mathcal{B}]e^{s\mathcal{A}} ds\right)^j W_0. 
\end{equation}
Since $\sum_{j=1}^\infty \frac{e^{t\mathcal{A}}}{\tau^{j}}C(t)_{j}$ is convergent in $H^1\times L^2(\mathbb{T}^d)$ (as one deduce e.g. from \eqref{eq:C-fac}), we obtain that
\begin{align*}
\left(e^{\tau\mathcal{A}}+\sum_{j=1}^\infty \frac{e^{\tau\mathcal{A}}}{\tau^{j}}C(\tau)_{j}\right) W_0\to &\left(\sum_{j=0}^\infty \frac{(\frac{\psi^2}{2}[[\mathcal{A},\mathcal{B}],\mathcal{B}])^j}{j!}\right)W_0\\
=&\left(I+\frac{\psi^2}{2}[[\mathcal{A},\mathcal{B}],\mathcal{B}]\right)W_0\\
=&\begin{pmatrix}
w_0\\
\dot{w}_0-\psi^2w_0
\end{pmatrix},\quad \tau\to 0,
\end{align*}
in $H^1\times L^2(\mathbb{T}^d)$ (where we used that $[[\mathcal{A},\mathcal{B}],\mathcal{B}]^2=0$). Finally, being also $\sum_{j=1}^\infty\sum_{i=1}^j\frac{e^{t\mathcal{A}}}{\tau^{j-\frac{i}{2}}}R(t)_{j,i}$ convergent in $H^1\times L^2(\mathbb{T}^d)$ and since there exists $c>0$ such that $\|R(\tau)_{i,j}\|_{H_1\times L^2(\mathbb{T}^d)}\leq c\tau^j$, we obtain that
$$\left\|\sum_{j=1}^\infty\sum_{i=1}^j\frac{e^{\tau\mathcal{A}}}{\tau^{j-\frac{i}{2}}}R(\tau)_{i,j}\right\|_{H^1\times L^2(\mathbb{T}^d)}\to 0,\quad \tau\to 0, $$
which concludes the proof.
\end{proof}

\section{Density of saturation space}\label{sec:saturation}
We consider the $(2d+1)$-dimensional vector subspace of $L^2(\mathbb{T}^d)$ spanned by the spatial control functions:
$$\mathcal{H}_0:={\rm span}\{1,\cos(e_1x),\sin(e_1x),\dots,\cos(e_dx),\sin(e_dx)\}, $$
and define $\mathcal{H}_j, j\geq 1,$ as the largest vector space whose elements $\phi$ can be written as 
$$\phi=\phi_0-\sum_{i=1}^N \phi_i^2,\quad \phi_0,\dots,\phi_n\in\mathcal{H}_{j-1},\quad N\in\mathbb{N}. $$
We moreover define the saturation space as $\mathcal{H}_\infty=\cup_{j=0}^\infty \mathcal{H}_j$.
\begin{lemma}\label{lem:density}
The vector space $\mathcal{H}_\infty$ is dense in $L^2(\mathbb{T}^d)$.
\end{lemma}
A more general version of this Lemma appeared in \cite[Proposition 2.5]{nersesyan-parabolic} for the study of nonlinear parabolic equations with additive controls. The proof here is simpler, and we furnish it for completeness.
\begin{proof}
We prove the statement by showing that 
$${\rm span}\{\cos(kx),\sin(kx)\mid k\in \mathbb{Z}^d\}\subset \mathcal{H}_\infty.$$
From 
$$\cos(2kx)=1-2\sin^2(kx),\quad -\cos(2kx)=1-2\cos^2(kx) $$
we deduce that $\pm \cos(2e_jx)\in \mathcal{H}_1$, and also that $\pm\cos^2(e_jx),\pm\sin^2(e_jx)\in\mathcal{H}_1$, for all $j=1,\dots,d$.
From 
$$\pm \sin(2kx)=1-(\sin(kx)\mp \cos(kx))^2 $$
we deduce that $\pm \sin(2e_jx)\in \mathcal{H}_1$ for all $j=1,\dots,d$.
Finally, from 
$$\pm\cos((k+m)x)=1-\frac{1}{2}(\cos(kx)\mp \cos(mx))^2-\frac{1}{2} (\sin(kx)\pm \sin(mx))^2 $$
we thus deduce that $\pm\cos(kx)\in\mathcal{H}_\infty$ for all $k\in\mathbb{Z}^d$, and from
$$\pm\sin((k+m)x)=1-\frac{1}{2}(\sin(kx)\mp\cos(mx))^2-\frac{1}{2}(\cos(kx)\mp\sin(mx))^2 $$
we thus deduce that $\pm\sin(kx)\in\mathcal{H}_\infty$ for all $k\in\mathbb{Z}^d$, concluding the proof.

\end{proof}


\section{Small-time global approximate controllability of the velocity}\label{sec:velocity}
In this section we prove that we can globally control the velocity in small times without changing the profile, approximately.
\begin{proposition}\label{prop:main-tool}
Consider an initial state $(w_0,\dot{w}_0)\in H^1\times L^2(\mathbb{T}^d)$ satisfying \eqref{eq:initial-condition1},
and any final velocity $ f\in L^2(\mathbb{T}^d)$. Then, for any error and time $\varepsilon,T>0$ there exist a smaller time $\tau\in[0,T)$ and a piecewise constant control law $p:[0,\tau]\to \mathbb{R}^{2d+1}$ such that the solution $w$ of \eqref{eq:wave} associated with the control \eqref{eq:low-modes},\eqref{eq:cos-sin} with initial condition $\left(w(t=0),\frac{\partial}{\partial t}w(t=0)\right)=(w_0,\dot{w}_0)$ satisfies
 $$\left\|\left(w(\cdot,\tau),\frac{\partial}{\partial t}w(\cdot,\tau)\right)-(w_0,f)\right\|_{H^1\times L^2(\mathbb{T}^d)}< \varepsilon. $$
\end{proposition}
The core of the proof of Proposition \ref{prop:main-tool} is given in the following Lemma.

\begin{lemma}\label{lemmino}
Let $(w_0,\dot{w}_0)\in H^1\times L^2(\mathbb{T}^d)$ and $\phi\in L^2(\mathbb{T}^d)$. Then, for any $\varepsilon,T>0$, there exist $\tau\in[0,T)$ and $p:[0,\tau]\to \mathbb{R}^{2d+1}$ piecewise constant such that the solution $w$ of \eqref{eq:wave} associated with the control \eqref{eq:low-modes},\eqref{eq:cos-sin} and with the initial condition $\left(w(t=0),\frac{\partial}{\partial t}w(t=0)\right)=(w_0,\dot{w}_0)$ satisfies
 $$\left\|\begin{pmatrix}w(\cdot,\tau)\\ \frac{\partial}{\partial t}w(\cdot,\tau)\end{pmatrix}-\begin{pmatrix}w_0\\ \dot{w}_0+w_0\phi\end{pmatrix}\right\|_{H^1\times L^2(\mathbb{T}^d)}< \varepsilon. $$
\end{lemma}
Leu us show how Proposition \ref{prop:main-tool} follows from Lemma \ref{lemmino}

\begin{proof}[Proof of Proposition \ref{prop:main-tool}]
Let $f$ be the final velocity in the statement, and define
\begin{equation}\label{eq:auxiliary-function}
\phi_\varepsilon=\frac{f-\dot{w}_0}{w_0}\chi_{\mathbb{T}^d\setminus Z_\varepsilon(w_0)}\in L^2(\mathbb{T}^d),
\end{equation}
where
$$ Z_\varepsilon(w_0):=\{x\in\mathbb{T}^d\mid {\rm dist}(x,Z(w_0))<\varepsilon\},\quad Z(w_0):=\{z\in\mathbb{T}^d\mid w_0(z)=0\},  $$
and $\chi_{S}$ is the characteristic function of any subset $S\subset \mathbb{T}^d$. Then one has
$$\|(\dot{w}_0+w_0\phi_\varepsilon)-f\|_{L^2(\mathbb{T}^d)}\leq \|\dot{w}_0\|_{L^2(Z_\varepsilon(w_0))}+\|f\|_{L^2(Z_\varepsilon(w_0))}\to 0, \quad \text{as }\varepsilon\to 0, $$
thanks to the fact that, since by hypothesis (cf. \eqref{eq:initial-condition1}) $w_0$ is supported only on a finite set of Fourier modes, the Lebesgue measure of $Z_\varepsilon(w_0)$ tends to zero when $\varepsilon$ tends to zero. Hence, Proposition \ref{prop:main-tool} follows from Lemma \ref{lemmino} by choosing $\phi=\phi_\varepsilon$ with $\varepsilon$ small enough. 
\end{proof}
Before proving Lemma \ref{lemmino}, let us recall that the concatenation $q*p$ of two scalar control laws $p:[0,T_1]\to \mathbb{R},q:[0,T_2]\to \mathbb{R} $ is the scalar control law defined on $[0,T_1+T_2]$ as follows
$$(q*p)(t)=\begin{cases}
p(t), & t\in[0,T_1]\\
q(t-T_1), & t\in(T_1,T_1+T_2],
\end{cases} $$
and the definition extends to controls with values in $\mathbb{R}^{2d+1}$ componentwise. Denoting with $\mathcal{R}(t,W_0,p)$ the solution $(w(t),\frac{\partial}{\partial t}w(t))$ of \eqref{eq:bilinear} at time $t$, associated with a control $p$ and the initial condition $W_0=(w_0,\dot{w}_0)$, we will often use the fact that
$$\mathcal{R}(T_1+t,W_0,q*p)=\mathcal{R}(t,\mathcal{R}(T_1,W_0,p),q),\quad t>0. $$
\begin{proof}[Proof of Lemma \ref{lemmino}]
Let us start by assuming that the following property holds for any $n\in\mathbb{N}$:
\begin{itemize}
\item[($P_n$)] for any $(w_0,\dot{w}_0)\in H^1\times L^2(\mathbb{T}^d)$, $\phi\in\mathcal{H}_n$, and any $\varepsilon,T>0$, there exist $\tau\in[0,T)$ and $p:[0,\tau]\to \mathbb{R}^{2d+1}$ piecewise constant such that the solution $w$ of \eqref{eq:wave} associated with the control \eqref{eq:low-modes} and with the initial condition $(w_0,\dot{w}_0)$ satisfies
 $$\left\|\begin{pmatrix}w(\cdot,\tau)\\ \frac{\partial}{\partial t}w(\cdot,\tau)\end{pmatrix}-e^{\phi\mathcal{B}}\begin{pmatrix}w_0\\ \dot{w}_0\end{pmatrix}\right\|_{H^1\times L^2(\mathbb{T}^d)}< \varepsilon. $$
\end{itemize}
Noticing that, thanks to the fact that $\mathcal{B}^2=0$, one has
$$e^{\phi\mathcal{B}}\begin{pmatrix}w_0\\ \dot{w}_0\end{pmatrix}=\begin{pmatrix}w_0\\ \dot{w}_0+w_0\phi\end{pmatrix},$$
the property $(P_n)$ combined with the density property proved in Lemma \ref{lem:density} implies at once Lemma \ref{lemmino}.

We are thus left to prove the property $(P_n)$. An analogous property appeared in \cite[Theorem 2.2]{duca-nersesyan} in the study of nonlinear Schr\"odinger equations with bilinear control. We furnish the proof also here for completeness; we do it by induction on $n$. \\

\textbf{Basis of induction: $n=0$}\\
If $\phi\in\mathcal{H}_0$, there exists $(p_0,\dots,p_{2d})\in\mathbb{R}^{2d+1}$ such that $\phi(x)=\sum_{i=0}^{2d}p_j\mu_j(x) $. Consider then the solution of \eqref{eq:bilinear} associated with the constant control $p^\tau:=(p_0,\dots,p_{2d})/\tau\in\mathbb{R}^{2d+1}$ and with the initial condition $(w_0,\dot{w}_0)$, that is,
$$
\mathcal{R}(t,\begin{pmatrix}
w_0\\
\dot{w}_0
\end{pmatrix},p^\tau)=\exp\left(t \left(\mathcal{A}+\frac{\phi}{\tau}\mathcal{B}\right)\right)\begin{pmatrix}
w_0\\
\dot{w}_0
\end{pmatrix}.$$
Applying \eqref{eq:limit1} with $\xi=\phi$, we find $\tau\in[0,T)$ such that
$$\left\|\mathcal{R}(\tau,\begin{pmatrix}
w_0\\
\dot{w}_0
\end{pmatrix},p^\tau)-e^{\phi\mathcal{B}}\begin{pmatrix}
w_0\\
\dot{w}_0
\end{pmatrix}\right\|_{H^1\times L^2(\mathbb{T}^d)}<\varepsilon, $$ 
which proves the desired property.\\

\textbf{Inductive step: $n\Rightarrow n+1$}\\
Assuming that $(P_n)$ holds, we prove $(P_{n+1})$. If $\phi\in\mathcal{H}_{n+1}$, there exist $N\in \mathbb{N}$ and $\phi_0,\dots,\phi_N\in\mathcal{H}_n$ such that $\phi=\phi_0-\sum_{i=1}^N \phi_i^2$. Consider, e.g., $\phi_1$: thanks to \eqref{eq:limit2}, we can fix $\gamma\in[0,T/3)$ small enough such that
$$\left\|e^{-\gamma^{-1/2}\phi_1\mathcal{B}}e^{\gamma\mathcal{A}}e^{\gamma^{-1/2}\phi_1\mathcal{B}}\begin{pmatrix}
w_0\\
\dot{w}_0
\end{pmatrix}-e^{-\phi_1^2\mathcal{B}}\begin{pmatrix}
w_0\\
\dot{w}_0
\end{pmatrix}\right\|_{H^1\times L^2(\mathbb{T}^d)}< \varepsilon/2.  $$

Thanks to the inductive hypothesis, for any $\epsilon,T,\gamma>0$ there exist $\delta\in[0,T/3)$ and a piecewise constant control $p^{\delta,\gamma}:[0,\delta]\to \mathbb{R}^{2d+1}$ such that 
\begin{equation}\label{eq:first-impulsion}
 \left\|
\mathcal{R}(\delta,\begin{pmatrix}
w_0\\
\dot{w}_0
\end{pmatrix},p^{\delta,\gamma})-e^{\gamma^{-1/2}\phi_1\mathcal{B}}\begin{pmatrix}
w_0\\
\dot{w}_0
\end{pmatrix}\right\|_{H^1\times L^2(\mathbb{T}^d)}< \epsilon. 
\end{equation}
Consider now a zero control $0|_{[0,\gamma]}=(0,\dots,0):[0,\gamma]\to\mathbb{R}^{2d+1}$ (that is, a free evolution), applied on a time interval of size $\gamma$: thanks to \eqref{eq:first-impulsion} and the fact that, for any $t\in\mathbb{R}$, $e^{t\mathcal{A}}$ is bounded, there exists $C=C(\gamma)$ such that
\begin{align*}
& \left\|
\mathcal{R}(\delta+\gamma,\begin{pmatrix}
w_0\\
\dot{w}_0
\end{pmatrix},0|_{[0,\gamma]}*p^{\delta,\gamma})-e^{\gamma\mathcal{A}}e^{\gamma^{-1/2}\phi_1\mathcal{B}}\begin{pmatrix}
w_0\\
\dot{w}_0
\end{pmatrix}\right\|_{H^1\times L^2(\mathbb{T}^d)}\\
=&\left\|
e^{\gamma\mathcal{A}}\mathcal{R}(\delta,\begin{pmatrix}
w_0\\
\dot{w}_0
\end{pmatrix},p^{\delta,\gamma})-e^{\gamma\mathcal{A}}e^{\gamma^{-1/2}\phi_1\mathcal{B}}\begin{pmatrix}
w_0\\
\dot{w}_0
\end{pmatrix}\right\|_{H^1\times L^2(\mathbb{T}^d)}< C\epsilon.  
\end{align*}
Now, we use again the inductive hypothesis to deduce that there exist $\delta'\in[0,T/3)$ and a piecewise constant control $p^{\delta',\gamma}:[0,\delta']\to \mathbb{R}^{2d+1}$ such that
$$\left\|\mathcal{R}(\delta',e^{\gamma\mathcal{A}}e^{\gamma^{-1/2}\phi_1\mathcal{B}}\begin{pmatrix}
w_0\\
\dot{w}_0
\end{pmatrix},p^{\delta',\gamma})-e^{-\gamma^{-1/2}\phi_1\mathcal{B}}e^{\gamma\mathcal{A}}e^{\gamma^{-1/2}\phi_1\mathcal{B}}\begin{pmatrix}
w_0\\
\dot{w}_0
\end{pmatrix}\right\|_{H^1\times L^2(\mathbb{T}^d)}<\epsilon. $$
Then, thanks to \eqref{eq:continuity}, there exists $C'=C'(\|p^{\delta',\gamma}\|_{L^1},\delta')$ such that
\begin{align*}
& \left\|
\mathcal{R}(\delta+\gamma+\delta',\begin{pmatrix}
w_0\\
\dot{w}_0
\end{pmatrix},p^{\delta',\gamma}*0|_{[0,\gamma]}*p^{\delta,\gamma})-e^{-\phi^2_1\mathcal{B}}\begin{pmatrix}
w_0\\
\dot{w}_0
\end{pmatrix}\right\|_{H^1\times L^2(\mathbb{T}^d)}\\
\leq & \left\| \mathcal{R}(\delta',\mathcal{R}(\delta+\gamma,\begin{pmatrix}
w_0\\
\dot{w}_0
\end{pmatrix},0|_{[0,\gamma]}*p^{\delta,\gamma}),p^{\delta',\gamma})-\mathcal{R}(\delta',e^{\gamma\mathcal{A}}e^{\gamma^{-1/2}\phi_1\mathcal{B}}\begin{pmatrix}
w_0\\
\dot{w}_0
\end{pmatrix},p^{\delta',\gamma}) \right\|_{H^1\times L^2(\mathbb{T}^d)}\\
+&\left\|\mathcal{R}(\delta',e^{\gamma\mathcal{A}}e^{\gamma^{-1/2}\phi_1\mathcal{B}}\begin{pmatrix}
w_0\\
\dot{w}_0
\end{pmatrix},p^{\delta',\gamma})-e^{-\gamma^{-1/2}\phi_1\mathcal{B}}e^{\gamma\mathcal{A}}e^{\gamma^{-1/2}\phi_1\mathcal{B}}\begin{pmatrix}
w_0\\
\dot{w}_0
\end{pmatrix} \right\|_{H^1\times L^2(\mathbb{T}^d)}\\
+&\left\|e^{-\gamma^{-1/2}\phi_1\mathcal{B}}e^{\gamma\mathcal{A}}e^{\gamma^{-1/2}\phi_1\mathcal{B}}\begin{pmatrix}
w_0\\
\dot{w}_0
\end{pmatrix}-e^{-\phi^2_1\mathcal{B}}\begin{pmatrix}
w_0\\
\dot{w}_0
\end{pmatrix} \right\|_{H^1\times L^2(\mathbb{T}^d)}\\
\leq & C'C\epsilon+\epsilon+\varepsilon/2.
\end{align*}

Choosing $\epsilon>0$ small enough such that $C'C\epsilon+\epsilon<\varepsilon/2$, we have then proved that the piecewise constant control $p^{\delta',\gamma}*0|_{[0,\gamma]}*p^{\delta,\gamma}$ steers the initial state $(w_0,\dot{w}_0)$ $\varepsilon$-close to the state 
$$\begin{pmatrix}
w_0\\
\dot{w}_0-\phi_1^2 w_0
\end{pmatrix}=e^{-\phi_1^2\mathcal{B}}\begin{pmatrix}
w_0\\
\dot{w}_0
\end{pmatrix}$$ 
in time $\tau:=\delta'+\gamma+\delta<T$. We can now repeat the same argument w.r.t. $\phi_2$ (arguing as if we were starting from the initial state $(w_0,\dot{w}_0-\phi_1^2w_0)$) and prove that the system can be steered arbitrarily close to the state $(w_0,\dot{w}_0-\phi_1^2w_0-\phi_2^2w_0)$ in arbitrarily small times, and iteratively to $(w_0,\dot{w}_0-\sum_{i=1}^N\phi_i^2w_0)$. To conclude, by inductive hypothesis, there exists a piecewise contant control $p$ steering the state $(w_0,\dot{w}_0-\sum_{i=1}^N\phi_i^2w_0)$ arbitrarily close to the state
$$e^{\phi_0\mathcal{B}}\begin{pmatrix}w_0\\ \dot{w}_0-\sum_{i=1}^N\phi_i^2w_0 \end{pmatrix}=e^{(\phi_0-\sum_{i=1}^N\phi_i^2)\mathcal{B}}\begin{pmatrix}w_0\\ \dot{w}_0\end{pmatrix}=e^{\phi\mathcal{B}} \begin{pmatrix}w_0\\ \dot{w}_0\end{pmatrix}$$
in arbitrarily small times, which ends the proof of the property $(P_n)$.
\end{proof}

\section{Small-time global approximate controllability}\label{sec:proof}
In this section, we start by proving the small-time version of Theorem \ref{thm:main-result}.
\begin{theorem}\label{thm:small-time}
Consider an initial state $(0,0)\neq(w_0,\dot{w}_0)\in H^1\times L^2(\mathbb{T}^d)$ such that \eqref{eq:initial-condition1}, or \eqref{eq:initial-condition2}, holds. Then, for any final state $(w_1,\dot{w}_1)\in H^1\times L^2(\mathbb{T}^d)$ and any error and time $\varepsilon,T>0$, there exist a smaller time $\tau\in[0,T)$ a piecewise constant control law $p:[0,\tau]\to \mathbb{R}^{2d+1}$ such that the solution $w$ of \eqref{eq:wave} associated with the control \eqref{eq:low-modes},\eqref{eq:cos-sin} and with the initial condition $\left(w(t=0),\frac{\partial}{\partial t}w(t=0)\right)=(w_0,\dot{w}_0)$ satisfies
 $$\left\|\left(w(\cdot,\tau),\frac{\partial}{\partial t}w(\cdot,\tau)\right)-(w_1,\dot{w}_1)\right\|_{H^1\times L^2(\mathbb{T}^d)}< \varepsilon. $$
\end{theorem}
\begin{proof}
We first remark that:
\begin{itemize}
\item If the final profile $w_1$ is such that $w_1=0$, we consider $\widetilde{w}_1=\epsilon\ll1$: the approximate controllability towards $\widetilde{w}_1$ implies, thanks to the triangular inequality, the approximate controllability towards $w_1$.

\item Moreover, for any final profile $w_1$ and any $\epsilon>0$, we consider a finite set $\mathcal{K}_\epsilon\subset \mathbb{Z}^d$ such that
$$\left\|\sum_{k\in\mathcal{K}_\epsilon}\langle w_1,\varphi_k\rangle\varphi_k-w_1\right\|_{H^1(\mathbb{T}^d)}<\epsilon. $$
Then, thanks to the triangular inequality, the approximate controllability in $H^1\times L^2$ towards $(\sum_{k\in\mathcal{K}_\epsilon}\langle w_1,\varphi_k\rangle\varphi_k,\dot{w}_1)$ implies the approximate controllability in $H^1\times L^2$ towards $(w_1,\dot{w}_1)$ if $\epsilon$ is small enough.
\end{itemize}
 We are thus left to show Theorem \ref{thm:main-result} w.r.t. any final profiles $w_1\neq 0$ supported on a finite set of Fourier modes. 
We first show it under the assumption \eqref{eq:initial-condition1}. We do it by combining three steps, which can be outlined as
$$\begin{pmatrix}
w_0\\
\dot{w}_0
\end{pmatrix}\xrightarrow{(i)} \begin{pmatrix}
w_0\\
f
\end{pmatrix}\xrightarrow{(ii)} \begin{pmatrix}
w_1\\
g
\end{pmatrix} \xrightarrow{(iii)} \begin{pmatrix}
w_1\\
\dot{w}_1
\end{pmatrix}, $$
and roughly read as follows:
\begin{itemize}
\item[$(i)$] thanks to Proposition \ref{prop:main-tool}, we can change in small times the initial velocity $\dot{w}_0$ of the wave profile into an arbitrary velocity $f$, without changing the initial profile $w_0$;
\item[$(ii)$]thanks to Proposition \ref{lem:moment-problem}, we can choose the velocity $f$ in $(i)$ to be such that the associated free evolution sends the initial profile $w_0$ into the final profile $w_1$ in arbitrarily small times;
\item[$(iii)$] finally, we use again Proposition \ref{prop:main-tool} to change in arbitrarily small times the freely evolved velocity $g$ into the final velocity $\dot{w}_1$, without changing the final profile $w_1$. 
\end{itemize}



More precisely: let $f\in L^2(\mathbb{T}^d)$ be such that the solution to \eqref{eq:wave} with initial condition $(w_0,f)$ and identically zero control on a time interval of size $\gamma<T/3$ satisfies $w(\gamma)=w_1$ (the existence of such $f$ is guaranteed by Proposition \ref{lem:moment-problem}). Denote moreover with $g:=\frac{\partial}{\partial t} w(\gamma)$ the freely evolved velocity. Thanks to Proposition \ref{prop:main-tool}, for any $\epsilon,T>0$ there exists a time $\delta\in[0,T/3)$ and a piecewise constant control $p:[0,\delta]\to \mathbb{R}^{2d+1}$ such that
$$
 \left\|
\mathcal{R}(\delta,\begin{pmatrix}
w_0\\
\dot{w}_0
\end{pmatrix},p)-\begin{pmatrix}
w_0\\
f
\end{pmatrix}\right\|_{H^1\times L^2(\mathbb{T}^d)}<\epsilon.  
$$
Since $e^{t\mathcal{A}}$ is bounded for any $t$, there exists $C=C(\gamma)$ such that 
\begin{align*}
& \left\|
\mathcal{R}(\delta+\gamma,\begin{pmatrix}
w_0\\
\dot{w}_0
\end{pmatrix},0|_{[0,\gamma]}*p)-\begin{pmatrix}
w_1\\
g
\end{pmatrix}\right\|_{H^1\times L^2(\mathbb{T}^d)}\\
=&\left\|e^{\gamma\mathcal{A}}\mathcal{R}(\delta,\begin{pmatrix}
w_0\\
\dot{w}_0
\end{pmatrix},p)-e^{\gamma\mathcal{A}}\begin{pmatrix}
w_0\\
f
\end{pmatrix}\right\|_{H^1\times L^2(\mathbb{T}^d)}<C \epsilon.
\end{align*}
Now, thanks to Proposition \ref{prop:main-tool}, there exists a time $\delta'\in[0,T/3)$ and a piecewise constant control $p':[0,\delta']\to \mathbb{R}^{2d+1}$ such that 
$$
 \left\|
\mathcal{R}(\delta',\begin{pmatrix}
w_1\\
g
\end{pmatrix},p')-\begin{pmatrix}
w_1\\
\dot{w}_1
\end{pmatrix}\right\|_{H^1\times L^2(\mathbb{T}^d)}< \epsilon.  
$$
Then, thanks to \eqref{eq:continuity}, there exists $C'=C'(\|p'\|_{L^1},\delta')$ such that
\begin{align*}
& \left\|\mathcal{R}(\delta+\gamma+\delta',\begin{pmatrix}
w_0\\
\dot{w}_0
\end{pmatrix},p'*0|_{[0,\gamma]}*p)-\begin{pmatrix}
w_1\\
\dot{w}_1
\end{pmatrix}\right\|_{H^1\times L^2(\mathbb{T}^d)} \\
\leq  & \left\|\mathcal{R}(\delta',\mathcal{R}(\delta+\gamma,\begin{pmatrix}
w_0\\
\dot{w}_0
\end{pmatrix},0|_{[0,\gamma]}*p),p')-\mathcal{R}(\delta',\begin{pmatrix}
w_1\\
g
\end{pmatrix},p')\right\|_{H^1\times L^2(\mathbb{T}^d)}\\
+& \left\|\mathcal{R}(\delta',\begin{pmatrix}
w_1\\
g
\end{pmatrix},p')-\begin{pmatrix}
w_1\\
\dot{w}_1
\end{pmatrix}\right\|_{H^1\times L^2(\mathbb{T}^d)}\\
\leq & C'C\epsilon+\epsilon
\end{align*}
Taking $\epsilon$ small enough such that $C'C\epsilon+\epsilon<\varepsilon$, we have found a piecewise constant control $p'*0|_{[0,\gamma]}*p$ steering $(w_0,\dot{w}_0)$ $\varepsilon$-close to $(w_1,\dot{w}_1)$ in $H^1\times L^2(\mathbb{T}^d)$, in time $\tau:=\delta+\gamma+\delta'<T$. This concludes the proof under assumption \eqref{eq:initial-condition1}.

To prove the theorem under assumption \eqref{eq:initial-condition2}, we first consider a free evolution for an arbitrarily small time $\gamma>0$: thanks to \eqref{eq:free-dyn1}, $w(\gamma)\neq 0$ and $\langle w(\gamma),e^{ikx}\rangle=0$ for all but a finite set of $k\in\mathbb{Z}^d$. Hence, we can now consider as initial profile $w(\gamma)$ satisfying assumption \eqref{eq:initial-condition1}, and the previous argument applies.
\end{proof}
We conclude by showing how Theorem \ref{thm:main-result} follows from Theorem \ref{thm:small-time}.
\begin{proof}[Proof of Theorem \ref{thm:main-result}]
Thanks to Theorem \ref{thm:small-time}, for every error $\epsilon>0$ there exist times $\tau,\tau'\in[0,T/2)$ and piecewise constant controls $q:[0,\tau]\to \mathbb{R}^{2d+1}, q':[0,\tau']\to \mathbb{R}^{2d+1}$ such that
$$\left\|\mathcal{R}(\tau,\begin{pmatrix}
w_0\\
\dot{w}_0
\end{pmatrix},q)-\begin{pmatrix}
1\\
0
\end{pmatrix}\right\|<\epsilon,\quad \left\|\mathcal{R}(\tau',\begin{pmatrix}
1\\
0
\end{pmatrix},q')-\begin{pmatrix}
w_1\\
\dot{w}_1
\end{pmatrix}\right\|<\epsilon. $$

We now use the fact that $(w(t),\frac{\partial}{\partial t}w(t))\equiv(1,0)$ is a solution of \eqref{eq:bilinear} w.r.t. the control $0|_{[0,T-(\tau+\tau')]}$ (that is, the zero control applied on a time interval of size $T-(\tau+\tau')$), which gives
\begin{align*}
& \left\|\mathcal{R}(T,\begin{pmatrix}
w_0\\
\dot{w}_0
\end{pmatrix},q'*0|_{[0,T-(\tau+\tau')]}*q)-\begin{pmatrix}
w_1\\
\dot{w}_1
\end{pmatrix}\right\|_{H^1\times L^2(\mathbb{T}^d)} \\
\leq  & \left\|\mathcal{R}(\tau',\mathcal{R}(T-\tau',\begin{pmatrix}
w_0\\
\dot{w}_0
\end{pmatrix},0|_{[0,T-(\tau+\tau')]}*q),q')-\mathcal{R}(\tau',\begin{pmatrix}
1\\
0
\end{pmatrix},q')\right\|_{H^1\times L^2(\mathbb{T}^d)}\\
+& \left\|\mathcal{R}(\tau',\begin{pmatrix}
1\\
0
\end{pmatrix},q')-\begin{pmatrix}
w_1\\
\dot{w}_1
\end{pmatrix}\right\|_{H^1\times L^2(\mathbb{T}^d)}\\
\leq &C'\!\left\|\mathcal{R}(T-(\tau+\tau'),\mathcal{R}(\tau,\begin{pmatrix}
w_0\\
\dot{w}_0
\end{pmatrix},q),0|_{[0,T-(\tau+\tau')]})\!-\!\mathcal{R}(T-(\tau+\tau'),\begin{pmatrix}
1\\
0
\end{pmatrix},0|_{[0,T-(\tau+\tau')]})\right\|+\epsilon \\ 
\leq & C'C\epsilon+\epsilon,
\end{align*}
where the existence of $C'=C'(\|q'\|_{L^1},\tau')$ is given by \eqref{eq:continuity} and the existence of $C=C(T-(\tau+\tau'))$ follows from the boundedness of $e^{t\mathcal{A}}$ for any fixed $t$. Taking $\epsilon$ small enough such that $C'C\epsilon+\epsilon<\varepsilon$, we have found a piecewise constant control $q'*0|_{[0,T-(\tau+\tau')]}*q$ steering $(w_0,\dot{w}_0)$ $\varepsilon$-close to $(w_1,\dot{w}_1)$ in $H^1\times L^2(\mathbb{T}^d)$, in time $T$. This concludes the proof.
\end{proof}

\textbf{Acknowledgements}

The author is thankful to Thomas Chambrion, Sylvain Ervedoza, Vahagn Nersesyan, Mario Sigalotti, and Marius Tucsnak for helpful discussions.

This work has been supported by the STARS Consolidator Grant 2021 “NewSRG” of the University of Padova, and by the PNRR MUR project PE0000023-NQSTI.

\bibliographystyle{spmpsci}
\bibliography{references}

 \end{document}